\documentclass[12pt]{amsart}

\usepackage[utf8]{inputenc}
\usepackage{amsmath,amsthm, url, comment}
\usepackage[normalem]{ulem}

\usepackage{colonequals}
\usepackage[alphabetic]{amsrefs}
\usepackage{enumitem, comment}
\usepackage{amssymb}
\usepackage{xcolor}
\usepackage[colorlinks=true, pdfstartview=FitH, linkcolor=blue, citecolor=blue, urlcolor=blue]{hyperref}

\newtheorem{theorem}{Theorem}[section]
\newtheorem{lemma}[theorem]{Lemma}
\newtheorem{proposition}[theorem]{Proposition}

\theoremstyle{definition}

\newtheorem{example}[theorem]{Example}
\newtheorem{question}[theorem]{Question}

\theoremstyle{remark}
\newtheorem{remark}[theorem]{Remark}
\numberwithin{equation}{section}

\newcommand{\Gal}{\operatorname{Gal}}

\makeatletter
\@namedef{subjclassname@2020}{%
  \textup{2020} Mathematics Subject Classification}
  \makeatother

\begin{document}

\title{Linear system of hypersurfaces passing through a Galois orbit}

\author{Shamil Asgarli}
\address{Department of Mathematics and Computer Science \\ Santa Clara University \\ 500 El Camino Real \\ USA 95053}
\email{sasgarli@scu.edu}

\author{Dragos Ghioca}
\address{Department of Mathematics, University of British Columbia, Vancouver, BC V6T 1Z2}
\email{dghioca@math.ubc.ca}

\author{Zinovy Reichstein}
\address{Department of Mathematics, University of British Columbia, Vancouver, BC V6T 1Z2}
\email{reichst@math.ubc.ca}

\subjclass[2020]{Primary 14N05; Secondary 14J70, 14G15}
\keywords{linear system, hypersurface, finite fields, irreducibility}

\begin{abstract}
Let $d$ and $n$ be positive integers, and $E/F$ be a separable field extension of degree $m=\binom{n+d}{n}$. We show that if $|F| > 2$, then
there exists a point $P\in \mathbb{P}^n(E)$ which does not lie on any degree $d$ hypersurface defined over $F$. In other words, the $m$ Galois
conjugates of $P$ impose independent conditions on the $m$-dimensional $F$-vector space of degree $d$ forms in $x_0, x_1, \ldots, x_n$. As an application, we determine the maximal dimensions of linear systems $\mathcal{L}_1$ and $\mathcal{L}_2$  of hypersurfaces in $\mathbb P^n$ over a finite field $F$, where every $F$-member of $\mathcal{L}_1$ is reducible and every $F$-member of $\mathcal{L}_2$ is irreducible.
\end{abstract}

\maketitle

\section{Introduction}\label{sect:intro}
Consider the vector space $V$ of all degree $d$ homogeneous forms in $n+1$ variables with coefficients in a field $F$. An elementary counting argument shows that \[  \dim(V) = \binom{n+d}{n} .\] 
Let us denote this number by $m$. An $F$-point of $\mathbb{P}(V)$ can be identified with a projective hypersurface in $\mathbb{P}^n$ defined over $F$. It is well known that if $F$ is an infinite field, $l$ points of $\mathbb{P}^n(F)$ in general position impose linearly independent conditions on hypersurfaces of degree $d$, provided that $l \leqslant m$; cf.~Lemma~\ref{lem.split-case}.
In particular, for points $P_1, \ldots, P_m$ of $\mathbb{P}^n$ in general position, no hypersurface of degree $d$ passes through all of them. 

Suppose $F$ is an arbitrary field (possibly finite) and $E/F$ is a separable field extension of degree $m$. 
Can we choose $P \in \mathbb{P}^n(E)$ so that the $m$ Galois conjugates of $P$ impose independent conditions on degree $d$ hypersurfaces in $\mathbb{P}^n$? In other words, is there always a $P \in \mathbb{P}^n(E)$ which does not lie on any 
degree $d$ hypersurface defined over $F$? Our main result gives an affirmative answer to this question under a mild restriction on $F$.

\begin{theorem}\label{thm:main}
Let $d$ and $n$ be positive integers, and
$E/F$ be a separable field extension of degree $m  \colonequals \binom{n+d}{n}$. Assume that $|F|>2$. Then there exists a point $P\in\mathbb{P}^n(E)$ such that $P$ does not lie on any hypersurface of degree $d$ defined over $F$.
\end{theorem}

Theorem~\ref{thm:main} can be restated as follows: there exist $a_0, a_1, ..., a_n \in E$ such that the $m$ elements $a_0^{i_0}a_1^{i_1} \cdots a_n^{i_n}$ of $E$ are linearly independent over $F$.
Here $i_0, i_1, \ldots, i_n$ range over non-negative integers such that $i_0 + i_1 + \ldots +i_n = d$. Note that in the case, where $n = 1$, this assertion specializes to the Primitive Element Theorem for the separable field extension $E/F$. 

As an application of Theorem~\ref{thm:main}, we determine the maximal dimensions of linear systems $\mathcal{L}_1$ and $\mathcal{L}_2$  of hypersurfaces in $\mathbb P^n$ over a finite field $F$, where every $F$-member of $\mathcal{L}_1$ is reducible and every $F$-member of $\mathcal{L}_2$ is irreducible. Our main result in this direction is Theorem~\ref{thm:irreducible} below. Before stating it, we recall some terminology.

Let $F$ be a field. An $F$-linear system $\mathcal{L}$ of degree $d$ hypersurfaces in $\mathbb P^n$ is a linear subspace of such hypersurfaces defined over $F$. By the no-name lemma~\cite{Shafarevich}*{Appendix 3}, $\mathcal{L}$ has a basis $f_0, f_1, \ldots, f_r$ such that each $f_i$ is defined over $F$. Members of $\mathcal{L}$ are then hypersurfaces in $\mathbb P^n$ of the form $c_0 f_0 + \ldots + c_r f_r = 0$ where $c_0, \ldots, c_r$ are scalars. Members of $\mathcal{L}$ corresponding to $c_0, \ldots, c_r \in F$ are called \emph{$F$-members}. The \emph{dimension} of $\mathcal{L}$ is $r$ (the projective dimension).

Given a property $\mathcal{P}$ of algebraic hypersurfaces defined over a finite field $\mathbb F_q$, it is natural to ask the following.

\begin{question}\label{q:property-P}
What is the largest dimension of a linear system $\mathcal{L}$ of degree $d$ hypersurfaces in $\mathbb{P}^n$  such that every $\mathbb{F}_q$-member of $\mathcal{L}$ satisfies $\mathcal{P}$?
\end{question}

In our previous paper~\cite{AGR23}, we addressed Question~\ref{q:property-P} when $\mathcal{P}$ is
the property of being smooth. In the paper ~\cite{AGY23}, the first two authors and Chi Hoi Yip addressed Question~\ref{q:property-P}  when $\mathcal{P}$ is the property of being non-blocking\footnote{Here a hypersurface $X$ in $\mathbb P^n$ defined over $\mathbb{F}_q$ is called \emph{blocking} if $X \cap L$
has an $\mathbb F_q$-point for every line $L \subset \mathbb{P}^n$ defined over $\mathbb{F}_q$ and \emph{non-blocking} otherwise.}. Parts (a) and (b) of Theorem~\ref{thm:irreducible} below answer Question~\ref{q:property-P} when $\mathcal{P}$ is the property of being
reducible, and parts (c) and (d) when $\mathcal{P}$ is the property of being irreducible. 

\begin{theorem}\label{thm:irreducible}
Let $d \geqslant 2$ and $n \geqslant 1$ be integers, $m \colonequals \binom{n+d}{n}$,  
$r \colonequals \binom{n+d-1}{n}$,
and $\mathbb F_q$ be a finite field of order $q > 2$.
Then

\smallskip
(a) there exists an $(r-1)$-dimensional $\mathbb F_q$-linear system $\mathcal{L}_{\rm red}$ of degree $d$ hypersurfaces in $\mathbb P^n$ such that every $\mathbb{F}_q$-member of $\mathcal{L}_{\rm red}$ is reducible over $\mathbb{F}_q$.

\smallskip
(b) Every $\mathbb{F}_q$-linear system $\mathcal{L}$ of dimension $\geqslant r$ has an $\mathbb{F}_q$-member which is irreducible over $\mathbb{F}_q$.

\smallskip
(c) There exists an $(m - 1 - r)$-dimensional $\mathbb{F}_q$-linear system $\mathcal{L}_{\rm irr}$ of degree $d$ hypersurfaces in $\mathbb P^n$ such that every $\mathbb{F}_q$-member of $\mathcal{L}_{\rm irr}$ is irreducible over $\mathbb{F}_q$.

\smallskip
(d) Every $\mathbb F_q$-linear system $\mathcal{L}$ of dimension $\geqslant m - r$ has an $\mathbb{F}_q$-member which is reducible over $\mathbb{F}_q$.
\end{theorem}

When the finite field $\mathbb{F}_q$ is replaced by its algebraic closure $\overline{\mathbb{F}_q}$ or any other algebraically closed field, parts (a) and (b) of Theorem~\ref{thm:irreducible} remain valid, whereas the dimensions in parts (c) and (d) get reduced by $n$; see Proposition~\ref{prop:irreducible}. In particular, part (c) fails when $\mathbb F_q$ is replaced by an algebraically closed field.

Computer experiments with specific values of $n$ and $d$ suggest that the assertion of
Theorem~\ref{thm:main} may be true when $|F| = 2$,
even though our proof does not go through 
in this case. If the assumption that $|F| > 2$ can be dropped in Theorem~\ref{thm:main},  
then the assumption that $q > 2$ can be dropped in Theorem~\ref{thm:irreducible}.

The remainder of this paper is structured as follows. In Section~\ref{sect.infinite}, we use a general position argument to prove Theorem~\ref{thm:main} under the assumption that $F$ is infinite. In 
the case where $F$ is finite, the concept of general position no longer applies. Here we employ a point-counting argument. The strategy 
behind this counting argument is outlined in Section~\ref{sect.strategy}, and is carried out in Sections~\ref{sec:comb},~\ref{sect.count}~and~\ref{sect.finite}. 
In Section~\ref{sect:irreducible-linear-systems} we deduce Theorem~\ref{thm:irreducible} from Theorem~\ref{thm:main}. 
In Section~\ref{sect:irred-alg-closed} we prove 
a variant of Theorem~\ref{thm:irreducible} with $\mathbb F_q$ replaced by an algebraically closed field.

\medskip 

\textbf{Acknowledgements.} The second and third authors are supported by NSERC Discovery grants RGPIN-2018-03690 for D. Ghioca and RGPIN-2023-03353 for Z. Reichstein.  We are grateful to anonymous referees whose comments and questions helped us improve the exposition.

\medskip
\textbf{Data availability statement.} No data sets were used or generated in the course of this research.

\medskip
\textbf{Conflict of interest statement.} All authors have no conflicts of interest.

\section{Proof of Theorem~\ref{thm:main} in the case of infinite fields}
\label{sect.infinite} 

The following lemma is well known; we include a short proof for the sake of completeness.

\begin{lemma} \label{lem.split-case}
Let $F$ be an infinite field, $d$ and $n$ be positive integers, and $\displaystyle m =  \binom{n+d}{n}$. Then there exist 
$P_1, \ldots, P_m \in \mathbb{P}^n(F)$ such that no degree $d$ hypersurface in $\mathbb{P}^n$ passes through $P_1, \ldots, P_m$.
\end{lemma}

\begin{proof} Let $V_0 = H^0(\mathbb{P}^n, \mathcal{O}(d))$ be the $m$-dimensional vector space space of all degree $d$ forms in $x_0, \ldots, x_n$ 
and $V_{i} \subset V$ be the subspace of forms vanishing at $P_1, \ldots, P_i$.
Clearly $V_i \subseteq V_{i-1}$ for any choice of $P_1, \ldots, P_i$.
Requiring forms to vanish on each $P_i$ imposes one linear condition; 
hence, $\dim(V_i) \geqslant m - i$, again for any choice of $P_1, \ldots, P_i$.
We claim that for a suitable choice of $P_1, \ldots, P_m$, we have 
\begin{equation} \label{e.inclusion}
V_{i} \subsetneq V_{i-1} 
\end{equation}
for every $i = 1, 2, \ldots, m$ or equivalently, $\dim(V_i) = m - i$.  In particular,
for this choice of $P_1, \ldots, P_m$, we will have $\dim(V_m) = 0$, and the lemma will follow. 

We will choose $P_1, \ldots, P_i$ so that~\eqref{e.inclusion} holds, by induction on $i \in \{ 1, \ldots, m \}$. Indeed, assume $P_1, \ldots, P_{i-1}$ 
have been chosen. Since $\dim(V_{i-1}) \geqslant m - i + 1 > 0$, there exists a non-zero element $f_i \in V_{i-1}$. 
We will now choose $P_i \in \mathbb{P}^n(F)$ so that $f_i(P_i) \neq 0$. 
A point $P_i$ with this property exists since $F$ is an infinite field. 
For this choice of $P_i$, $f \in V_{i-1} \setminus V_i$, and \eqref{e.inclusion} follows.
This completes the proof of the claim and thus of Lemma~\ref{lem.split-case}.
\end{proof}

\begin{proposition} \label{prop.main-infinite} Let $d$ and $n$ be positive integers and $E/F$ be a commutative algebra of degree $m = \binom{n+d}{n}$ over $F$.
View $E$ as an $m$-dimensional vector space over $F$. Then there is a homogeneous polynomial function $H$ on 
the affine space $\mathbb A_F^{n+1}(E) \simeq \mathbb A_F^{(n+1)m}$ defined over $F$ with
the following property: For any field extension $F'/F$, $E' = E \otimes_F F'$, a point $a = (a_0: \ldots: a_n) \in  \mathbb P^n(E')$ 
lies on a hypersurface of degree $d$ defined over $F'$ if and only if $H(a_0, a_1, \ldots, a_n) = 0$.
\end{proposition}

\begin{proof}  Let $M_1, \ldots, M_m$ be distinct monomials of degree $d$ in $x_0, \ldots, x_n$. Clearly
$a = (a_0:a_1: \ldots : a_n) \in\mathbb P^n(E)$ lies on a hypersurface of degree $d$ in $\mathbb P^n$ defined over $F$ if and only if 
$M_1(a), \ldots, M_m(a)$ are linearly dependent over $F$.

Suppose $\{b_1, \ldots, b_n\}$ is an $F$-basis of $E$. Write 
\begin{equation} \label{e.structure-constants}
b_i b_j = \sum_{h = 1}^n c_{ij}^h b_h,
\end{equation}
where the structure constants $c_{ij}^h$ lie in $F$. Using the basis $b_1, \ldots, b_m$ we can identify $E$ with $F^m$ as an $F$-vector space
(not necessarily as an algebra). Set 
\begin{equation} \label{e.e_i-coordinates}
a_i = y_{i, 1} b_1 + \ldots + y_{i, m} b_m, 
\end{equation}
where each $y_{i, j} \in F$. 
Using formulas \eqref{e.structure-constants}, for every $s = 1, \ldots, m$, we can express $M_s(a)$ 
in the form $M_s(a) = p_{s, 1} b_1 + \ldots + p_{s, m} b_m$, 
where each $p_{s, t}$ is a homogeneous polynomial 
of degree $d$ in $y_{i, j}$ with coefficients in $F$. 
By abuse of notation, we will denote these polynomials by $p_{s, t}(y_{i, j})$.

Now, view $y_{i, j}$ as independent $(n+1)m$ variables, as $i$ ranges from $0$ to $n$ and $j$ ranges from $1$ to $m$.
Set
\[ H(y_{i, j}) = \det \, \begin{pmatrix}
p_{1,1}(y_{i, j}) & p_{1,2}(y_{i, j}) & \cdots & p_{1,m}(y_{i, j}) \\
p_{2,1}(y_{i, j}) & p_{2,2}(y_{i, j}) & \cdots & p_{2,m}(y_{i, j}) \\
\vdots  & \vdots  & \ddots & \vdots  \\
p_{m,1}(y_{i, j}) & p_{m,2}(y_{i, j}) & \cdots & p_{m,m}(y_{i, j}) 
\end{pmatrix} . \]
For any field extension $F'/F$, an $F'$-point $(\alpha_{i, j}') \in \mathbb A_F^{(n+1)m}$
represents a point $a' = (a_0': \ldots: a_m') \in \mathbb P^n(E')$, where
$a_i' = \alpha_{i, 1} b_1 + \ldots + \alpha_{i, m} b_m \in E'$ for each $i = 0, 1, \ldots, n$. 
By our construction, $H(\alpha_{i, j}) = 0$ if and only if $M_1(a'), \ldots, M_m(a')$ are linearly dependent over $F'$,
and the proposition follows.
\end{proof}

\begin{remark} In the case, where $E/F$ is a separable field extension of degree $m$, we can give an alternative description of $H$ as follows. Denote the
normal closure of $E$ over $F$ by $E^{\rm norm}$, the Galois group $\operatorname{Gal}(E^{\rm norm}/F)$ by $G$, and the Galois group
$\operatorname{Gal}(E^{\rm norm}/E)$ by $G_0$. Note that $[G:G_0] = [E: F] = m$. 

It is easy to see that there exists a homogeneous polynomial 
\[ P_{d, n} \in \mathbb Z[x_{ij} \, | \,  i =  1, \ldots, m \quad \text{and} \quad j = 0, 1,  \ldots, n  ] \]
such that $m$ points $(x_{i0}: \ldots: x_{in})$ of $\mathbb P^n$, where $i = 1, \ldots, m$, all lie on the same hypersurface of degree $d$ if and only if $P_{d, n}(x_{ij}) = 0$. Then given a point $A = (a_0, \ldots, a_n)$ in $\mathbb A_E^{n+1}$, we set
$H(A) = P_{d,n}(\sigma_1(A), \ldots, \sigma_m(A))$, where $\sigma_1, \ldots, \sigma_m$
are representatives of the $m$ left cosets of $G_0$ in $G$.
\end{remark} 

\begin{proof}[Conclusion of the proof of Theorem~\ref{thm:main}, assuming $F$ is an infinite field]
Let $H(y_{i,j})$ be the homogeneous polynomial function on $\mathbb A_F(E^n) \simeq \mathbb A_F^{(n+1)m}$ 
defined over $F$ whose existence is asserted by Proposition~\ref{prop.main-infinite}. We claim that $H$ is not identically $0$. 

Once this claim is established, Theorem~\ref{thm:main} readily follows from Proposition~\ref{prop.main-infinite}; since $F$ is an infinite field,
we can specialize each $x_{ij}$ to some $c_{ij} \in F$ so that $H(c_{ij}) \neq 0$. 

To prove the claim, it suffices to show that $H(c_{ij}) \neq 0$, for some choice of $c_{ij}$ in a larger field $F'$ containing $F$. 
Let us choose $F'$ so that $F'$ splits $E/F$, i.e., $E \otimes_F F'$ isomorphic to $E' \colonequals F' \times \ldots \times F'$ ($m$ times).
In particular, we can take $F'$ to be an algebraic closure of $F$. 

Using Proposition~\ref{prop.main-infinite}, we can rephrase the above observation as follows: to prove the existence of a point
$a = (a_0: a_1: \ldots: a_n) \in \mathbb P^n(E)$ with the property that it does not lie of any hypersurface of degree 
$d$ defined over $F$, it suffices to prove the existence of a point $a' = (a_0': \ldots: a_n) \in \mathbb P^n(E^{'})$
which does not lie on any hypersurface of degree $d$ defined over $F'$. To finish the proof, observe that the existence of $a'$ with this property
is equivalent to Lemma~\ref{lem.split-case} with $F = F'$.
\end{proof}

\section{Proof strategy for Theorem~\ref{thm:main} in the finite field case}
 \label{sect.strategy}

From now on, we will assume that $F = \mathbb{F}_q$ and $E = \mathbb{F}_{q^m}$ are finite fields. This section
outlines a strategy for a proof of Theorem~\ref{thm:main} in this case. We begin by proving Theorem~\ref{thm:main} 
under the assumption $q > d$, which greatly simplifies our counting argument.

\begin{proposition}\label{prop:q>d}
Let $q$ be a prime power, $d, n\in\mathbb{N}$ and $m\colonequals \binom{n+d}{n}$. Assume $q>d$. Then there exists a point $P\in\mathbb{P}^n(\mathbb{F}_{q^m})$ such that $P$ does not lie on any hypersurface of degree $d$ defined over $\mathbb{F}_q$.
\end{proposition}

Note that here $q = 2$ is allowed, unlike in Theorem~\ref{thm:main}, but only in the (trivial) case, where $d = 1$.
For the remainder of the paper, 
\[ \text{$\mathcal{H} \subset \mathbb P^n$ will denote the union of all hypersurfaces of degree $d$ defined over $\mathbb{F}_q$.}
\]

\begin{proof}[Proof of Proposition~\ref{prop:q>d}] Observe that $\deg(\mathcal{H}) = d(q^{m-1}+...+q+1)$. Since $q>d$, we have
$$
\deg(\mathcal{H})\leqslant (q-1)(q^{m-1}+\cdots+q+1) = q^{m} - 1 
$$
On the other hand, the degree of a space-filling hypersurface in $\mathbb P^n(\mathbb F_{q^m})$ defined over $\mathbb{F}_q$ is at least $q^{m}+1$; see, e.g.,~\cite{MR98}*{Th\'eor\`eme~2.1}. We conclude that $\mathcal{H}$ is not space-filling in $\mathbb P^n(\mathbb F_{q^m})$, and the proposition follows. 
\end{proof}

When $d \geqslant q$, we will need a more delicate argument to show that $\mathcal{H}$ does not contain every 
$\mathbb F_{q^m}$-point of $\mathbb P^n$. We will estimate
the number of $\mathbb F_{q^m}$-points on $\mathcal{H}$, with the goal of showing that this number
is strictly smaller than the number of $\mathbb F_{q^m}$-points in
$\mathbb P^n$. To estimate the number of $\mathbb{F}_{q^m}$-points on $\mathcal{H}$, we will subdivide
the hypersurfaces $X\subset \mathbb{P}^n$ of degree $d$ defined over $\mathbb{F}_q$ into two classes:

\begin{enumerate}[label=\alph*)]
\item \label{geom-irred-case} $X$ is geometrically irreducible (that is, irreducible over $\overline{\mathbb{F}_q}$), or
\item \label{not-geom-irred-case} $X$ is geometrically reducible. 
\end{enumerate}

When $X\subset\mathbb{P}^n$ is geometrically irreducible, we will use the inequality
\begin{equation} \label{e.C-M}
|X(\mathbb{F}_{q^m})| \leqslant (q^{m(n-1)} + \cdots + q^m + 1) + (d-1)(d-2) q^{m(n-3/2)} + 5 d^{13/3} q^{m(n-2)},
\end{equation}
due to Cafure and Matera~\cite{CM06}*{Theorem~5.2}. When $X$ is geometrically reducible, we will use
Serre's estimate~\cite{Ser91}*{Th\'eor\`eme},  
\begin{equation} \label{e.serre}
|X(\mathbb{F}_q)| \leqslant d q^{m(n-1)} + q^{m(n-2)} + \cdots + q^m + 1.
\end{equation}
Note that both of these are polynomial bounds in $q$ of degree $m(n-1)$. However, 
the one in Case~\ref{not-geom-irred-case} is asymptotically weaker, because the leading term $q^{m(n-1)}$
comes with coefficient $1$ in~\eqref{e.C-M} and with coefficient $d$ in~\eqref{e.serre}. 
To get a strong upper bound on the number of $\mathbb F_{q^m}$-points on $\mathcal{H}$, we need to make sure
that Case \ref{not-geom-irred-case} does not occur too often. In other words, if we let
$t$ denote the fraction of hypersurfaces in $\mathbb P^n$ over $\mathbb{F}_q$ of fixed degree $d$ which are \emph{not} 
geometrically irreducible, then our first task is to bound $t$ from above. Note that $t$ depends on $q$, $d$ and $n$.

Poonen showed that $t \to 0$, as $d \to \infty$ and $q$ and $n$ remain fixed; see~\cite{Poo04}*{Proposition 2.7}.
This is not enough for our purposes. We will refine the inequalities from the proof of \cite{Poo04}*{Proposition 2.7}
to establish the following upper bound on $t$.

\begin{proposition}\label{prop:sharper-bounds} 
Let $t$ denote the fraction of hypersurfaces in $\mathbb{P}^{n}$ of degree $d$ over $\mathbb{F}_q$ 
that are geometrically reducible.  Assume that one of the following conditions holds: 
\begin{itemize}
\item $n= 2$, $d\geqslant 6$ and  $q\geqslant 3$; or
\item $n\geqslant 3$, $d\geqslant 3$ and $q\geqslant 3$.
\end{itemize} 
Then $(d-1) tq \leqslant 2$. 
\end{proposition}

We will prove Proposition~\ref{prop:sharper-bounds} in Section~\ref{sect.count}, then use it to complete the proof of Theorem~\ref{thm:main} in Section~\ref{sect.finite}. In Section~\ref{sec:comb} we gather several 
elementary inequalities involving binomial coefficients, which
will be used in our proofs.


\section{Combinatorial bounds}
\label{sec:comb}

Throughout this section, we let $q,d\geqslant 3$ and $n \geqslant 2$ be integers.
For each $i$ between $0$ and $d$, 
set 
\begin{equation}
\label{eq:N_i}
N_i = \binom{n+d}{d} - \binom{n+i}{n} - \binom{n+d-i}{n}.
\end{equation}

\begin{lemma}
\label{lem:1001}
Assume $2(i+1) \leqslant d$.
Then
\begin{itemize}
\item[(a)] $N_{i+1} - N_i \geqslant d-2i-1$; and
\item[(b)] $N_{i+1} - N_1 \geqslant d - 3$.
\end{itemize}
\end{lemma}

\begin{proof}
(a) Using Pascal's identity recursively, we rewrite $N_{i+1} - N_i$ as
\begin{align*}
N_{i + 1} - N_i & = \binom{n+d-i-1}{n-1} - \binom{n+i}{n-1} \\
& = \sum_{j=0}^{d-i} \binom{n-2+j}{n-2} - \sum_{j=0}^{i+1}\binom{n-2+j}{n-2} \\
& = \sum_{j=i+2}^{d-i} \binom{n-2+j}{n-2}
\end{align*}
The above sum has $(d-i) - (i-1) = d - 2i + 1 \geqslant 1$ terms by our assumption on $i$. Moreover, each term
$\geqslant 1$, so the sum is $\geqslant d - 2i + 1$, as desired.

\smallskip
(b) Write
$N_{i+1} - N_{1}  = (N_{i+1} - N_i) + (N_i - N_{i-1}) + \ldots + (N_2 - N_1)$.
Part~(a) tells us that each term in this sum is non-negative, and the last term, $N_2 - N_1$, is $\geqslant d - 3$.
Thus
\begin{equation}
\label{e.t1-2}
N_{i+1} - N_{1}  = (N_{i+1} - N_i) + (N_i - N_{i-1}) + \ldots + (N_2 - N_1) \geqslant N_2-N_1\geqslant d-3,
\end{equation}
as desired. 
\end{proof}


\begin{lemma} \label{lem.u1} Let $\displaystyle u_1:=\sum_{i=1}^{\lfloor d/2\rfloor} q^{-N_i}$, where $N_i$ is as in \eqref{eq:N_i}. Then

\smallskip
(a) $\displaystyle u_1 \leqslant \frac{29}{27} q^{2-d}$ if 
$n=2$, $q\geqslant 3$ and $d\geqslant 6$. 

\smallskip
(b) $\displaystyle u_1 \leqslant \frac{3}{2} 
q^{-\frac{n(n+d-1)}{2}+n+1}$ for all $n\geqslant  3$, $q\geqslant  3$, and $d\geqslant  3$.
\end{lemma}

\begin{proof} 
We first estimate $N_1$ from below. Note that we assume $n \geqslant 2$ throughout.
\begin{align}
\nonumber N_1 & = \binom{n+d-1}{n-1}-\binom{n+1}{n} = \binom{n+d-1}{d} - \binom{n+1}{1} \\
\nonumber  & = \frac{(n+d-1)(n+d-2)\cdots (n+1)n}{d!} - (n+1) \\
\label{e.t1-3} &  = (n+d-1)\cdot \left(\frac{n+d-2}{d}\right)\cdots \left(\frac{n+1}{3}\right)\cdot \frac{n}{2}  - (n+1) \\
\nonumber & \geqslant \frac{(n+d-1)n}{2} - (n+1).
\end{align}
Using this estimate in combination with Lemma~\ref{lem:1001}(b), we obtain:
\begin{align*}
u_1 \leqslant & q^{-N_1} \cdot \sum_{i=1}^{\lfloor d/2\rfloor} q^{-(N_i - N_1)} 
     \leqslant q^{-N_1} \left( 1 + \sum_{i=2}^{\lfloor d/2\rfloor} q^{-(d-3)} \right) \\
     & \leqslant q^{-N_1}\left(1 + \left(\frac{d}{2}-1\right) q^{3-d} \right) 
     \leqslant q^{-\frac{(n+d-1)n}{2} + (n+1)} \left(1 + \left(\frac{d}{2}-1\right) q^{3-d} \right).
\end{align*}
An elementary computation shows that for integers $d \geqslant  6$ and $q\geqslant  3$,
the expression $\displaystyle \left(1 + \left(\frac{d}{2}-1\right) q^{3-d} \right)$ is at most $\displaystyle \frac{29}{27}$. (This maximal value 
is attained when $q=3$ and $d=6$.) This completes the proof of part (a). 

Similarly, when $q\geqslant 3$ and $d\geqslant 3$, the maximal value of the expression $\left(1 + \left(\frac{d}{2}-1\right) q^{3-d} \right)$ is $\displaystyle \frac{3}{2}$. (This maximal value is attained when $q=3$ and $d=3$). This completes the proof of part (b). 
%
\end{proof}

\begin{lemma}
\label{lem:1002} 
For each divisor $e>1$ of $d$, set
$\displaystyle M_e := \binom{d+n}{n}-e\cdot \binom{d/e+n}{n}$.
Then 
$$
M_{e} \geqslant  \binom{e}{2}\binom{n}{2} \left(\frac{d}{e}\right)^2 - e + 1.
$$
for any $n\geqslant  2$, $q\geqslant  3$, $d\geqslant  3$. Here $e\mid d$, where $e>1$. 
\end{lemma}

\begin{proof}
Let $S = T \cup F$, where $T$ and $F$ are disjoint sets of cardinality $d$ and $n$, respectively.
The binomial coefficient $\displaystyle \binom{d+n}{n}$ counts the number of $n$-subsets of $S$. 

Partition $T$ as $T = T_1 \cup T_2 \cup \cdots \cup T_{e}$, where $|T_i|=d/e$ for each $i$, and set $S_i = T_i \cup F$. 
Note that $|S_i| = (d/e) + n$; hence, the binomial coefficient $\displaystyle \binom{d/e+n}{n}$ counts the number of $n$-subsets of $S_i$. 
It is also clear that the number of \emph{common} $n$-subsets of $S_i$ and $S_j$ for $i\neq j$ is exactly $1$, 
namely the $n$-set $F$. Thus, the total number of $n$-subsets arising from $S_1, S_2, \ldots, S_e$ is exactly:
\begin{align*}
    e\cdot \left(\binom{d/e + n}{n}  - 1\right) + 1 = e \cdot \binom{d/e + n}{n} - e + 1.
\end{align*}
Next, we construct additional $n$-subsets of $S$ that are not contained in any $S_k$. 
Fix integers $1 \leqslant i < j \leqslant e$. Choose elements $a\in T_i$ and $b\in T_j$ and
consider $n$-subsets of $S$ of the form
$$
\{a, b\} \cup E
$$
for some $(n-2)$-subset $E$ of $F$. By our construction, 
$\displaystyle \{a, \, b \} \cup E$ is not contained in $S_k$ for any $1\leqslant k\leqslant e$. 
The number of subsets of the form $\displaystyle \{a, \, b \} \cup E$ is equal to $(d/e)\cdot (d/e)\cdot \binom{n}{n-2}$ 
once $i$ and $j$ are fixed, because there are $d/e$ ways to choose $a$ in $T_i$, $d/e$ ways to choose $b$ in $T_j$,
and $\binom{n}{n-2}=\binom{n}{2}$ ways to choose an $(n-2)$-subset $E$ of $F$. 
Varying $(i, j)$ among the $\binom{e}{2}$ choices, we get a total contribution of 
$$
\binom{e}{2} \binom{n}{2} \left(\frac{d}{e}\right)^2 
$$
many distinct $n$-subsets of $S$ that do not arise as $n$-subsets of $S_k$ for any $1\leqslant k \leqslant e$. Consequently,
$$
\binom{d+n}{n} - \left(e\cdot \binom{d/e + n}{n} - e + 1\right) \geqslant  \binom{e}{2} \binom{n}{2} \left(\frac{d}{e}\right)^2,
$$
leading to the lower bound
$$
M_{e} = \binom{d+n}{n} - e\cdot \binom{d/e + n}{n} \geq \binom{e}{2} \binom{n}{2} \left(\frac{d}{e}\right)^2 - e + 1,
$$
as claimed in the conclusion of Lemma~\ref{lem:1002}. 
\end{proof}

We will also need the following lower bound for the integers $M_e$ defined in Lemma~\ref{lem:1002}.

\begin{lemma}
\label{lem:1003}
If $n\geqslant 2$ and $d,q\geqslant 3$, then  for each divisor $e>1$ of $d$, we have: 
\begin{equation}
\label{eq:ineq.M_e}
M_e \geqslant  \frac{1}{4}\binom{n}{2} d^2 - d + 1.
\end{equation}
\end{lemma}

\begin{proof}
The bound from \eqref{eq:ineq.M_e} follows from Lemma~\ref{lem:1002} using that 
$$
\binom{e}{2}\binom{n}{2}\left(\frac{d}{e}\right)^2 - e + 1 \geqslant  \left(\frac{1}{2} - \frac{1}{2e} \right) \binom{n}{2} d^2  - d + 1 \geqslant  \frac{1}{4}\binom{n}{2} d^2- d + 1
$$
since $d\geqslant  e\geqslant  2$.
\end{proof}


\begin{lemma} \label{lem.u2} Set $\displaystyle u_2 \colonequals \sum_{e|d, e>1} q^{-M_e}$.
If $n\geqslant  2$, $q\geqslant  3$, $d\geqslant  3$, then 
$$
u_2 \leqslant(d-1) q^{-\frac{1}{4}\binom{n}{2} d^2 + d -1}.$$
\end{lemma}

\begin{proof}
First, we note that the number of divisors $e$ of $d$ with $e>1$ is at most $d-1$. Thus the sum the right hand side of $\displaystyle u_2\colonequals \sum_{e|d, e>1} q^{-M_e}$
has at most $d-1$ terms. By Lemma~\ref{lem:1003}, each term $q^{-M_e}$ is at most 
$\displaystyle q^{- \, \frac{\, 1 \,}{4} \binom{n}{2} d^2 + d - 1}$. Lemma~\ref{lem.u2}
now tells us that
$$
u_2 \leqslant (d-1) q^{-\frac{1}{4}\binom{n}{2} d^2 + d -1},
$$
as desired.
\end{proof}

Finally, we set
\begin{equation}
\label{eq:upper-bound-on-t}
v_1 \colonequals \frac{3}{2} \, q^{-\frac{(n+d-1)n}{2} + (n+1)} +  (d-1) q^{-\frac{1}{4}\binom{n}{2} d^2 + d - 1},
\end{equation} 
when $n\geqslant 3$, $q\geqslant 3$ and $d\geqslant 3$,  and

\smallskip
\begin{equation}
\label{eq:upper-bound-on-t 2}
v_2 \colonequals \frac{29}{27}\,  q^{2-d} +  (d-1) q^{-\frac{1}{4} d^2 + d - 1},
\end{equation}
when $n=2$, $q\geqslant 3$ and $d\geqslant 6$.
We will establish next upper bounds for $v_2$ and $v_1$ (in this order).

\begin{lemma}
\label{lem:1004} 
For $n=2$, $q\geqslant  3$ and $d\geqslant  6$, we have $(d-1)qv_2 \leqslant 2$.
\end{lemma}

\begin{proof}
Using \eqref{eq:upper-bound-on-t 2}, we write
$$
(d-1) v_2 q = \Theta(q, d) \colonequals (d-1) \left(\frac{29}{27} q^{3-d} +  (d-1) q^{-\frac{1}{4} d^2 + d }\right).
$$
For $d\geqslant 6$, both exponents in $q^{3-d}$ and $q^{-\frac{1}{4} d^2 + d }$ are negative. 
This yields $\Theta(q, d)\leqslant \Theta(3, d)$ for $q\geqslant 3$. We now view $\Theta(3, d)$
as a function of $d$, as $d$ ranges over the interval $[6, \infty)$. On this interval $\Theta(3, d)$ 
achieves its maximum at $d=6$. Thus, 
$(d-1) tq \leqslant \Theta(3, 6)\approx 1.125$. In particular, $(d-1)tq\leqslant 2$. 
\end{proof}

\begin{lemma}
\label{lem:1005} Assume that $n\geqslant  3$, $q\geqslant  3$ and $d\geqslant  3$. Then $(d-1) v_1q \leqslant 2$.
\end{lemma}

\begin{proof}
We argue as in the proof of Lemma~\ref{lem:1004}. For $n\geqslant  3$, the definition of $v_1$ from~\eqref{eq:upper-bound-on-t} implies 
$$
v_1 \leqslant 1.5 q^{4-\frac{3}{2}(d+2)} +  (d-1) q^{-\frac{3}{4} d^2 + d - 1}.
$$
where we have substituted $n=3$ in \eqref{eq:upper-bound-on-t}. Consequently,
$$
(d-1) v_1q  \leqslant \Psi(q, d) \colonequals  (d-1) \left( 1.5 q^{5-\frac{3}{2}(d+2)} +  (d-1) q^{-\frac{3}{4} d^2 + d }\right)
$$
We have $\Psi(q, d)\leqslant \Psi(3, d)$ for $q\geqslant 3$. Viewing $\Psi(3, d)$ as a function of $d$ and letting $d$ range over the interval $[3, \infty)$, we see that $\Psi(3, d)$ achieves its maximum on this interval when $d=3$. Thus, $(d-1) tq \leqslant \Psi(3, 3)\approx 0.257$. In particular, $(d-1)v_1q\leqslant 2$, as desired. 
\end{proof}


\section{Proof of Proposition~\ref{prop:sharper-bounds}} 
\label{sect.count}

Following Poonen~\cite{Poo04}*{Proof of Proposition 2.7}, we will write
\begin{equation} \label{e.t} t = t_1 + t_2 
\end{equation}
and estimate $t_1$ and $t_2$ separately. Here 

\begin{itemize}
\item $t_1$ is the proportion of hypersurfaces of degree $d$ in $\mathbb P^n$ defined over $\mathbb{F}_q$, which are 
reducible over $\mathbb F_q$, and

\item $t_2$ is the proportion of hypersurfaces of degree $d$ in $\mathbb P^n$ defined over $\mathbb{F}_q$, which are irreducible 
over $\mathbb F_q$ but reducible over $\mathbb{F}_{q^e}$ for some integer $e > 1$, dividing $d$.
\end{itemize}

\begin{lemma} \label{lem.t1} 
(a) Assume $n=2$, $q\geqslant 3$ and $d\geqslant 6$. Then $\displaystyle t_1 \leqslant \frac{29}{27} q^{2-d}$.

\smallskip
(b) Assume $n\geqslant  3$, $q\geqslant  3$, and $d\geqslant  3$. Then $\displaystyle t_1\leqslant 
\frac{3}{2}q^{-\frac{n(n+d-1)}{2}+n+1}$.
\end{lemma}

\begin{proof} It is shown in the proof of \cite{Poo04}*{Proposition 2.7} that 
\begin{equation} \label{e.t1}
t_1 \leqslant \sum_{i=1}^{\lfloor d/2\rfloor} q^{-N_i},
\end{equation}
where $\displaystyle N_i = \binom{n+d}{d} - \binom{n+i}{n} - \binom{n+d-i}{n}$, as in~\eqref{eq:N_i}.
Parts (a) and (b) now follow from Lemma~\ref{lem.u1}(a) and (b), respectively. (Note that the right hand side of the inequality~\eqref{e.t1} is denoted by $u_1$ in the statement of Lemma~\ref{lem.u1}.)
\end{proof}

Next, we prove a lower bound on the proportion $t_2$ of hypersurfaces which are irreducible but not geometrically irreducible. 
 
\begin{lemma} \label{lem.t2} Let $n\geqslant  2$, $q\geqslant  3$, $d\geqslant  3$, we have
$\displaystyle
t_2 \leqslant(d-1) q^{-\frac{1}{4}\binom{n}{2} d^2 + d -1}$.
\end{lemma}

\begin{proof}
It is shown in the proof of \cite{Poo04}*{Proposition 2.7} that
\begin{equation} \label{e.t2}
t_2 \leqslant\sum_{e|d, e>1} q^{-M_e}
\end{equation}
where $\displaystyle M_e = \binom{d+n}{n}-e\binom{d/e+n}{n}$.  The desired conclusion now follows from Lemma~\ref{lem.u2}. (Note that the right hand side of the inequality~\eqref{e.t2} is denoted by $u_2$ in the statement of Lemma~\ref{lem.u2}.)
\end{proof}

We are finally ready to finish the proof of Proposition~\ref{prop:sharper-bounds}.

\begin{proof}[Proof of Proposition~\ref{prop:sharper-bounds}.]
Writing $t = t_1 + t_2$, as in~\eqref{e.t} and 
using Lemma~\ref{lem.t1} and Lemma~\ref{lem.t2}, we obtain
$$
t \leqslant \frac{3}{2} \,  q^{-\frac{(n+d-1)n}{2} + (n+1)} +  (d-1) q^{-\frac{1}{4}\binom{n}{2} d^2 + d - 1}
$$
when $n\geqslant 3$, $q\geqslant 3$ and $d\geqslant 3$, while 
$$
t \leqslant \frac{29}{27}\,  q^{2-d} +  (d-1) q^{-\frac{1}{4} d^2 + d - 1},
$$
when $n=2$, $q\geqslant 3$ and $d\geqslant 6$. Note that the right hand sides of these inequalities are precisely the quantities $v_1$ and $v_2$ from~\eqref{eq:upper-bound-on-t} 
and~\eqref{eq:upper-bound-on-t 2}.  
The desired conclusion, \[ (d-1)r q \leqslant 2, \] now follows from Lemmas~\ref{lem:1005}~and~\ref{lem:1004}, respectively.
\end{proof}

\section{Conclusion of the proof of Theorem~\ref{thm:main}}
\label{sect.finite}

The case when $F$ is infinite is examined in Section~\ref{sect.infinite}. Thus we will assume that $F = \mathbb{F}_q$ and $E = \mathbb{F}_{q^m}$ are finite fields. The case where $q > d$ is handled in Proposition~\ref{prop:q>d}. Hence, from now on, we assume that $q\leqslant d$.

We follow the strategy outlined in Section~\ref{sect.strategy}. Recall the notation we used there:

\smallskip
\begin{itemize}
\item $\mathcal{H}$ denotes the union of all degree $d$ hypersurfaces in $\mathbb{P}^n$ defined over $\mathbb{F}_q$, and

\smallskip
\item $t$ denotes the fraction of these hypersurfaces which are \emph{not} geometrically irreducible.  
\end{itemize}

\smallskip
Our goal is to show that there exists an $\mathbb F_{q^m}$-point in $\mathbb P^n$ which does not lie on $\mathcal{H}$.
As the total number of hypersurfaces of degree $d$ defined over $\mathbb{F}_q$ is $\displaystyle q^{m-1}+...+q+1 = \frac{q^m-1}{q-1}$, 
there are exactly $\displaystyle t\left(\frac{q^m-1}{q-1}\right)$ hypersurfaces of degree $d$ which are geometrically reducible. 
Using the upper bounds~\eqref{e.C-M} and~\eqref{e.serre} on the number of points of a hypersurface of degree $d$, 
we obtain the following inequality:
\begin{align*}
\# \mathcal{H}(\mathbb{F}_{q^m}) \leqslant & \left( \frac{q^{m}-1}{q-1} \right)\cdot ( (1-t)((q^{m(n-1)} + \cdots + q^m + 1) + (d-1)(d-2) q^{m(n-3/2)} \\ & + 5 d^{13/3} q^{m(n-2)}) + t (d q^{m(n-1)} + q^{m(n-2)} + \cdots + q^m + 1) ),
\end{align*}
where $m  \colonequals \binom{n+d}{n}$. After some cancellations, we can bound the term in the parenthesis after $\displaystyle \frac{q^m-1}{q-1}$ from above by
\begin{align}\label{eq:inside-parenthesis}
& (1+(d-1)t) q^{m(n-1)} + q^{m(n-2)} + ... + q^m+1 \\ 
& + (d-1)(d-2) q^{m(n-3/2)} + 5d^{13/3} q^{m(n-2).} \nonumber
\end{align}
By Proposition~\ref{prop:sharper-bounds}, we have 
\begin{equation} \label{e.prop6}
(d-1)t \leqslant \frac{2}{q},
\end{equation}
for all $n\geqslant 3$, $d\geqslant 3$ and $q\geqslant  3$, or $n=2$, $q\geqslant 3$ and $d\geqslant 6$. Since we already know  that Theorem~\ref{thm:main} holds when $q>d$ (see Proposition~\ref{prop:q>d}), we may assume that the inequality~\eqref{e.prop6}
holds unless $(n,q,d)$ equals $(2,3,3)$, $(2,3,4)$, $(2,3,5)$, $(2,4,4)$, $(2,4,5)$ and $(2,5,5)$. These exceptional cases will 
be handled using a computer at the end of the proof; we ignore them for now. Next, we bound the lower-order terms in the expression ~\eqref{eq:inside-parenthesis}.

\smallskip
\textbf{Claim.} If $n\geqslant  2$, $q\geqslant 3$ and $d\geqslant 3$, then we have  \begin{align*}
    (d-1)(d-2) q^{m(n-3/2)} + (q^{m(n-2)} + \cdots + q^m+1) + 5d^{13/3} q^{m(n-2)} <  q^{m(n-1)-1}
\end{align*} 
In order to verify this inequality, we first note that
\begin{equation}
\label{eq:lower_2}
q^{m(n-2)}+\cdots + q^m+1=\frac{q^{m(n-1)}-1}{q^m-1}< \frac{q^{m(n-1)}}{q^m-1}< \frac{q^{m(n-1)}}{1000q},
\end{equation}
since $q\geqslant 3$ and $m\geqslant (d+2)(d+1)/2\geqslant 10$ because $d\geqslant 3$. Employing \eqref{eq:lower_2}, we see that
the left-hand side of the inequality in the Claim is less than
\begin{equation}
\label{eq:lower_3}
(d-1)(d-2)q^{m(n-3/2)}+\frac{q^{m(n-1)-1}}{1000}+ 5d^{13/3}q^{m(n-2)}.
\end{equation}  
Dividing the expression from \eqref{eq:lower_3} by $q^{m(n-1)-1}$, we can easily check 
$$(d-1)(d-2)q^{1-m/2}+\frac{1}{1000}+5d^{13/3}q^{1-m}<1,$$
keeping in mind that $q\geqslant 3$ and $m\geqslant (d+2)(d+1)/2$, while $d\geqslant 3$. This completes the proof of the Claim.

Combining the Claim with the inequality~\eqref{e.prop6},
the quantity in \eqref{eq:inside-parenthesis} is less than
\begin{align*}
\left(1 + \frac{2}{q} \right) q^{m(n-1)} +  q^{m(n-1)-1} < q^{m(n-1)}  + 3 q^{m(n-1)-1}.
\end{align*}
Thus, we obtain the following upper bound on $\# \mathcal{H}(\mathbb{F}_{q^m})$.
$$
\# \mathcal{H}(\mathbb{F}_{q^m}) < \left(\frac{q^m-1}{q-1}\right) \left(q^{m(n-1)}  + 3 q^{m(n-1)-1}\right) 
$$
To show that $\mathcal{H}$ does not pass through every $\mathbb{F}_{q^m}$-point in $\mathbb{P}^n$, it is enough to show
that
$$
\left(\frac{q^m-1}{q-1}\right) \left(q^{m(n-1)}  + 3q^{m(n-1)-1}\right)  \leqslant q^{mn},
$$
because $\# \mathbb{P}^n(\mathbb{F}_{q^m}) = q^{mn}+\cdots+q^m+1$. By replacing $q^{m}-1$ with $q^m$ on the left-hand-side, we claim that the stronger inequality holds:
$$
q^{m} (q^{m(n-1)}+ 3q^{m(n-1)-1}) \leqslant q^{mn+1} - q^{mn}.
$$
After cancelling out $q^{mn-1}$ from both sides, it remains the show,
$$
q + 3 \leqslant q^{2} - q.
$$
This last inequality $q^2-2q-3 \geqslant  0$ is valid for all $q\geqslant  3$. Therefore, we have established Theorem~\ref{thm:main}
with $F = \mathbb F_q$ and $E = \mathbb F_{q^m}$, for all triples $(n, q, d)$ with $n\geqslant  2$, $q\geqslant  3$, $d \geqslant 1$, and  
$(n, q, d) \neq (2, 3, 3)$, $(2, 3, 4)$, $(2,3,5)$, $(2,4,4)$, $(2,4,5)$, $(2,5,5)$.

We now complete the proof of Theorem~\ref{thm:main} by a computer-assisted computation in these six exceptional cases. For each of the exceptional triples $(n, q, d)$, it suffices to find a single point $P\in \mathbb{P}^2(\mathbb{F}_{q^m})$ such that $P$ does not lie on any degree $d$ hypersurface defined over $\mathbb{F}_q$. Here $m=\binom{n+d}{n}$.

\medskip
When $(n, q, d) = (2, 3, 3)$ we write $\mathbb{F}_{3^{10}}$ as $\mathbb{F}_3[a]/(a^{10}+a^4+a+1)$, and check that $P=(a:a^8:1)$ does not lie on any cubic plane curve defined over $\mathbb{F}_3$.

When $(n, q, d) = (2, 3, 4)$, we write $\mathbb{F}_{3^{15}}$ as $\mathbb{F}_3[a]/(a^{15}+a^2-1)$
and check that $P=(a:a^9:1)$ does not lie on any quartic plane curve defined 
over $\mathbb{F}_{3}$. 

When $(n, q, d) = (2, 3, 5)$, we write $\mathbb{F}_{3^{21}}$ as $\mathbb{F}_3[a]/(a^{21}+a^{16}-1)$ and check that $P=(a:a^{18}:1)$ does not lie on any quintic plane curve defined over $\mathbb{F}_{3}$.

When $(n, q, d) = (2, 4, 4)$, we write $\mathbb{F}_{4^{15}}$ as $\mathbb{F}_{4}[a]/(a^{15}+a+1)$ and check that $P=(a^3:a^8:1)$ does not lie on any quartic plane curve defined over $\mathbb{F}_{4}$. 

When $(n, q, d) = (2, 4, 5)$, we write $\mathbb{F}_{4^{21}}$ as $\mathbb{F}_{4}[a]/(a^{21}+a^2+1)$ and check that $P=(a^6:a^{11}:1)$ does not lie on any quintic plane curve defined over $\mathbb{F}_{4}$. 

When $(n, q, d) = (2, 5, 5)$, we write $\mathbb{F}_{5^{21}}$ as $\mathbb{F}_{5}[a]/(a^{21}+a^{18}+a^{14}+1)$ and check that $P=(a:a^9:1)$ does not lie on any quintic plane curve defined over $\mathbb{F}_{5}$. \qed

\section{Proof of Theorem~\ref{thm:irreducible}}\label{sect:irreducible-linear-systems}

We will first construct the linear systems $\mathcal{L}_{\rm red}$ and $\mathcal{L}_{\rm irr}$ in parts (a) and (c), then use them to prove parts (b) and (d). 
We will use the notation from the statement of Theorem~\ref{thm:irreducible} throughout this section: $d$ and $n$ are positive integers,
$$m \colonequals \binom{n + d}{n} \quad \quad \text{and} \quad \quad
r \colonequals \binom{n + d - 1}{n}.$$

\medskip
(a) We take $\mathcal{L}_{\rm red}$ to be the linear system of hypersurfaces of degree $d$ in $\mathbb P^n$ containing a fixed hyperplane $H$. Let us say, $H$ is the hyperplane given by $x_0 = 0$. Then $\mathcal{L}_{\rm red}$ consists of polynomials of the form $x_0 F(x_0, x_1, \ldots, x_n)$, where $F(x_0, x_1, \ldots, x_n)$ is a polynomial of degree $d-1$ in $x_0, x_1, \ldots, x_n$. (Note that we are using the assumption that $d \geqslant 2$ to conclude that any polynomial of this form is reducible.) The dimension of $\mathcal{L}_{\rm red}$ is thus equal to the dimension of the linear system of homogeneous polynomials $F(x_1, \ldots, x_n)$ of degree $d-1$ in $x_1, \ldots, x_n$. In other words, $\dim(\mathcal{L}_{\rm red}) = r-1$.

\smallskip
(c) We apply Theorem~\ref{thm:main} for degree $d-1$ hypersurfaces in $\mathbb{P}^n$. Note that as we replace $d$ by $d-1$ in Theorem~\ref{thm:main}, $m$ gets replaced by $r$.
We obtain a point $P\in \mathbb{P}^n(\mathbb{F}_{q^r})$ 
that is not contained in any hypersurface of degree $d-1$ defined over $\mathbb{F}_q$. Clearly, $P$ is also not contained in any hypersurface of degree \emph{at most} $d-1$. Let $S=\{P_1, \cdots, P_r\}$  be the orbit of $P$ under $\Gal(\mathbb{F}_{q^r}/\mathbb{F}_q)$, where $P_1 = P$. Consider the vector space $V_{S}$ of degree $d$ forms defined over $\mathbb{F}_q$, which vanish at the point $P$ (and therefore at each point of $S$). Since vanishing at each additional point imposes at most one new linear condition, we obtain
$\operatorname{dim} V_{S} \geqslant  m  - r$.
Pick linearly independent forms $f_0, f_1, ..., f_{m-1-r} \in V_S$ and consider the $(m-1-r)$-dimensional linear system $\mathcal{L}_{\rm irr}=\langle f_0, f_1, ..., f_{m- 1 - r} \rangle$ of degree $d$ hypersurfaces. 

It remains to show that each $\mathbb{F}_q$-member of $\mathcal{L}_{\rm irr}$ is irreducible over $\mathbb{F}_q$. Indeed, assume the contrary: we factor $f$ as $f = g\cdot h$, where $g, h \in \mathbb F_q[x_0, \ldots, x_n]$ are homogeneous polynomials of degree at most $d-1$. Since $f(P)=0$, we have $g(P)=0$ or $h(P)=0$. This leads to a contradiction, because $P$ does not lie on a hypersurface in $\mathbb P^n$ of degree at most $d-1$ defined over $\mathbb{F}_q$. Thus, every $\mathbb F_q$-member of $\mathcal{L}_{\rm irr}$ is irreducible over $\mathbb{F}_q$.

\smallskip
(b) Suppose $\mathcal{L}$ is a linear system of hypersurfaces of degree $d$ in $\mathbb P^n$ of dimension $\displaystyle r$.
Then $\mathcal{L}$ and $\mathcal{L}_{\rm irr}$ intersect non-trivially in $\mathbb P^{m-1}$. An $\mathbb{F}_q$-member of $\mathcal{L}$ corresponding to the $\mathbb{F}_q$-point of intersection is irreducible over $\mathbb{F}_q$.

\smallskip
(d) Similarly, if $\mathcal{L}$ is a linear system of hypersurfaces of degree $d$ in $\mathbb P^n$ of dimension $\geqslant m - r$, then $\mathcal{L}$ and $\mathcal{L}_{\rm red}$ intersect non-trivially in $\mathbb P^{m-1}$.  An $\mathbb{F}_q$-member of $\mathcal{L}$ corresponding to an $\mathbb{F}_q$-point of intersection is reducible over $\mathbb{F}_q$.
\qed

\section{A variant of Theorem~\ref{thm:irreducible}
over an algebraically closed field}
\label{sect:irred-alg-closed}


In this section we prove a variant of Theorem~\ref{thm:irreducible}, where the finite field $\mathbb F_q$ is replaced by an algebraically closed field $F$. As we mentioned in the Introduction, parts (a) and (b) of Theorem~\ref{thm:irreducible} remain valid in this setting, whereas the dimensions in parts (c) and (d) get reduced by $n$.

\begin{proposition} \label{prop:irreducible}
Let $n, d \geqslant 2$ be integers, $m = \binom{n+d}{n}$,  $r = \binom{n+d-1}{n}$,
and $F$ be an algebraically closed field. 

\smallskip
(a) There exists an $(r-1)$-dimensional $F$-linear system $\mathcal{M}_{\rm red}$ of degree $d$ hypersurfaces in $\mathbb P^n$ such that every $F$-member of $\mathcal{L}_{\rm red}$ is reducible over $F$.

\smallskip
(b) Every $F$-linear system $\mathcal{L}$ of dimension $\geqslant r$ has an $F$-member which is irreducible over $F$.

\smallskip
(c) There exists an $(m - r - n-1)$-dimensional $F$-linear system $\mathcal{L}_{\rm irr}$ of degree $d$ hypersurfaces in $\mathbb P^n$ such that every ${F}$-member of $\mathcal{L}_{\rm irr}$ is irreducible.

\smallskip
(d) Let $\mathcal{L}$ be an $F$-linear system of degree $d$ hypersurfaces in $\mathbb P^n$. 
If $\dim(\mathcal{L}) \geqslant m - r - n$, then $\mathcal{L}$ has a reducible $F$-member.
\end{proposition}

\begin{proof}
(a) The construction of $\mathcal{L}_{\rm red}$ in the proof of Theorem~\ref{thm:irreducible}(a) goes through over an arbitrary field.

\smallskip
(b) Let $\mathcal{L} = \langle f_0, \ldots, f_{t} \rangle$ of degree $d$ hypersurfaces in $\mathbb P^n$ defined over $F$,
Let \[ f_{\lambda}(x_0, \ldots, x_n) = \lambda_0 f_0 + \ldots + \lambda_{t} f_t \] be the member of this system corresponding to $\lambda = (\lambda_0: \ldots:\lambda_{t}) \in \mathbb P^t$.
Assume that every $F$-element of $\mathcal{L}$ is reducible, that is, $f_{\lambda}$ is a reducible polynomial in $F[x_0, \ldots, x_n]$ for every $F$-point
$\lambda = (\lambda_0: \ldots: \lambda_{t}) \in \mathbb P^t(F)$. Our goal is to show that $\dim(\mathcal{L}) \leqslant r-1$. Let us consider two cases.

\smallskip
Case 1: The generic member of $\mathcal{L}$ is irreducible. Here by the generic member we mean the member
coresponding to the generic point of $\mathbb P^t$.  Equivalently, $f_{\lambda}$ is irreducible as a polynomial 
in $x_0, \ldots, x_n$ over the field $F(\lambda_0, \ldots, \lambda_{t})$. 

A description of the polynomials $f_{\lambda}$ that may occur in this case can be found in Schinzel's book~\cite{Sch00}*{Chapter 3, Theorem 37}.  
It follows from this description that if $\operatorname{char}(F)$ does not divide $d$, then
the maximal dimension of $\mathcal{L}$ is $d$, and is achieved by 
the linear system $\langle x_1^d, x_1^{d-1}x_2, x_1^{d-2}x_2^2, \ldots, x_2^d \rangle$. 
On the other hand, if $\operatorname{char}(F)$ divides $d$, then the maximal dimension of $\mathcal{L}$ is either $d$,
attained in the same way as above) or $\binom{n+d/p}{n}-1$. The latter is achieved by the linear system spanned by all monomials of the form $x_0^{p  \, i_0}x_1^{p \, i_1}\cdots x_n^{p \, i_n}$ with $i_0 + \ldots + i_n = d/p$.

It remains to show that (i) $d \leqslant r-1$ and (ii) if $p \geqslant 2$ divides $d$, then $\binom{n + d/p}{n} \leqslant r$. 
By Pascal's identity, for a fixed $d$, $\binom{n + d - 1}{n}$ increases with $n$. In particular, since
$n \geqslant 2$, we have
\[ \frac{(d+1)d}{2} = \binom{2 + d - 1}{2} \leqslant \binom{n + d -1}{n} = r. \]
Since $d \geqslant 2$, this yields $d = (d+1) - 1\leqslant \frac{(d + 1)d}{2} -1 \leqslant  r -1$,
proving (i). To prove (ii), note that $d/p \leqslant d-1$. Thus
\[ \binom{n+d/p}{n} \leqslant \binom{n+ d - 1}{n} = r, \]
as desired. 

\smallskip
Case 2: The generic member of $\mathcal{L}$ is reducible. Equivalently, $f_{\lambda}$ is reducible as a polynomial 
in $x_0, \ldots, x_n$ over $F(\lambda_0, \ldots, \lambda_{t})$. Using Gauss' Lemma, and the fact that
$f_{\lambda}$ is homogeneous of degree $1$ in $\lambda_0, \ldots, \lambda_{t}$, we see that  
\[ f_{\lambda}(x_0, \ldots, x_n) = g(x_0,...,x_n)\cdot h_{\lambda}(x_0,...,x_n), \]
where $g \in F[x_0, \ldots, x_n]$ is a homogeneous polynomial of degree $d_1$, 
$h_{\lambda} = \lambda_0 h_0 + \ldots + \lambda_{t} h_{t}$ for some homogeneous
polynomials $h_0, \ldots, h_{t} \in F[x_0, \ldots, x_n]$ of degree $d_2 \geqslant 1$ and
$d_1 + d_2 = d$. Here $h_0, \ldots, h_t$ are linearly independent over $F$. Thus
\[ \dim(\mathcal{L}) = t \leqslant \binom{n + d_2}{n} -1 \leqslant \binom{n + d -1}{n} -1 = r -1. \]
This completes the proof of part (b).

\smallskip
To prove (c) and (d), let $\mathcal{R}$ be the locus of reducible hypersurfaces inside the parameter space $\mathbb P^{m-1}$ of all degree $d$ hypersurfaces in $\mathbb{P}^n$.  
Denote the dimension 
of $\mathcal{R}$ by $s$. Then every linear subspace of (projective) dimension $\geqslant m - 1 - s$ intersects $\mathcal{R}$ in $\mathbb{P}^{m-1}$; on the other hand, a linear subspace of (projective) dimension $< m - 1 - s$ in general position will \emph{not} meet $\mathcal{R}$ in $\mathbb{P}^{m-1}$. Since $F$ is algebraically closed, a nonempty intersection always has an $F$-point. 
In other words, the following are equivalent:

\smallskip
\begin{itemize}
\item
every linear system of (projective) dimension $t$ has a reducible $F$-member,  and  

\smallskip
\item
$t \geqslant m - 1 - s$. 
\end{itemize}

\smallskip 
It remains to show that 
\begin{equation} \label{e.dim-of-R}
s = r + n - 1;
\end{equation}
this immediately implies both (c) and (d). To prove~\eqref{e.dim-of-R}, note that $\displaystyle \mathcal{R} = \bigcup_{i =1}^{\lfloor d/2 \rfloor} \mathcal{R}_{i}$, where $\mathcal{R}_i$ consists of reducible hypersurfaces $F(x_0, \ldots, x_n) = 0$, where $F = F_1\cdot F_2=0$ and $F_1, \, F_2$ are homogeneous polynomials 
in $x_0, x_1, \ldots, x_n$ of degree $i$ and $d-i$, respectively. In other words, $\mathcal{R}_i$ is the image of the map $\mathbb{P}^{m_1- 1}\times  \mathbb{P}^{m_2-1} \to \mathbb{P}^{m - 1}$ given by $(F_1, F_2)\to F_1\cdot F_2$ where $m_1 = \binom{n+i}{n}$, $m_2=\binom{n+d-i}{n}$. It is easy to see that 
\[ \dim(\mathcal{R}_{i}) = \displaystyle \binom{n+i}{n}+\binom{n+d-i}{n}-2. \]
The difference $\dim(\mathcal R_{i}) - \dim(\mathcal R_{i+1})$ is exactly the quantity $N_{i+1} - N_i$ we considered at the beginning of Section~\ref{sec:comb}; 
see~\eqref{eq:N_i}. By Lemma~\ref{lem:1001}(a), $N_{i+1} - N_i \geqslant 0$ whenever $2(i +1) \leqslant d$. 
We conclude that $\dim(\mathcal R_i)$ assumes its maximal value when $i=1$. In other words,  
\[ s = \dim(\mathcal{R}) = \dim(\mathcal{R}_1) = \binom{n+1}{n}+\binom{n+d-1}{n}-2 = \binom{n+d-1}{n} + n - 1 = r+ n - 1, \]
as claimed. 
\end{proof}
%

\begin{remark} \label{rem.harmless}
Note that the assumption that $d \geqslant 2$ in Theorem~\ref{thm:irreducible} and Proposition~\ref{prop:irreducible} is harmless, since every hypersurface of degree $1$ in $\mathbb P^n$ is irreducible. Moreover, over an algebraically closed field, every hypersurface of degree $d \geqslant 2$ in $\mathbb P^1$ is reducible. Thus the assumption that
$n \geqslant 2$ in the statement of Proposition~\ref{prop:irreducible} is harmless as well.
\end{remark}

\end{document}